\documentclass[11pt,leqno]{article}
\usepackage{amsmath} 
\usepackage[mathscr]{eucal}
\usepackage{amssymb} 
\usepackage{theorem}
\usepackage{xcolor}

\usepackage{enumerate}

\theoremstyle{change}  
\newtheorem{theorem}{Theorem}[section] 
\newtheorem{lemma}[theorem]{Lemma}  
\newtheorem{proposition}[theorem]{Proposition}
\newtheorem{corollary}[theorem]{Corollary}
\theorembodyfont{\rmfamily}  

\newtheorem{remark}[theorem]{Remark}

\newtheorem{notation}[theorem]{Notation}
\newtheorem{nothing}[theorem]{} 

\newenvironment{proof}{\noindent{\bf Proof}\ }{\qed\bigskip}

\renewcommand{\le}{\leqslant} 
\renewcommand{\ge}{\geqslant}

\newcommand{\abar}{\bar{a}}
\def\Aut{\mathrm{Aut}}
\newcommand{\bbar}{\bar{b}}

\newcommand{\calF}{\mathcal{F}}

\newcommand{\calP}{\mathcal{P}}

\newcommand{\catfont}{\mathsf}
\newcommand{\FF}{\mathbb{F}}

\newcommand{\GL}{\mathrm{GL}}
\newcommand{\id}{\mathrm{id}}

\newcommand{\Inn}{\mathrm{Inn}}

\newcommand{\lexp}[2]{\setbox0=\hbox{$#2$} \setbox1=\vbox to
                 \ht0{}\,\box1^{#1}\!#2}

\newcommand{\lMod}[1]{\llap{\phantom{|}}_{#1}\catfont{Mod}}

\newcommand{\Out}{\mathrm{Out}}

\newcommand{\PGL}{\mathrm{PGL}}

\newcommand{\PIM}{\mathrm{PIM}}
\newcommand{\PSL}{\mathrm{PSL}}
\newcommand{\Pu}{P\langle u\rangle}
\newcommand{\qed}{\nobreak\hfill
                  \vbox{\hrule\hbox{\vrule\hbox to 5pt
                  {\vbox to 8pt{\vfil}\hfil}\vrule}\hrule}}
\newcommand{\Qv}{Q\langle v \rangle}

\newcommand{\SL}{\mathrm{SL}}

\newcommand{\ZZ}{\mathbb{Z}}

\newcounter{parag}[section]

\title{Functorial equivalence classes of $2$-blocks of tame representation type}
\author{Robert Boltje, Serge Bouc, and Deniz Y\i lmaz}

\begin{document}
\sloppy
\maketitle

\begin{abstract} 
For any block of a finite group over an algebraically closed field of characteristic $2$ which has dihedral, semidihedral, or generalized quaternion defect groups, we determine explicitly the decomposition of the associated diagonal $p$-permutation functor over an algebraically closed field $\FF$ of characteristic $0$ into a direct sum of simple functors. As a consequence we see that two blocks with dihedral, semidihedral, or generalized quaternion defect groups are functorially equivalent over $\FF$ if and only if their fusion systems are isomorphic. It is an open question if two blocks (with arbitrary defect groups) that are functorially equivalent over $\FF$ must have isomorphic fusion systems. The converse is wrong in general.
\end{abstract}


\section{Introduction}
Let $p$ be a prime, $k$ an algebraically closed field of characteristic $p$, and $\FF$ an algebraically closed field of characteristic zero.

\smallskip
In \cite{BoucYilmaz2022}, the notion of a {\em diagonal $p$-permutation functor} over an arbitrary commutative ring $R$ was defined. It is by definition an $R$-linear functor from the category $R pp_k^\Delta$ to the category of $R$-modules. The objects of the category $R pp_k^\Delta$ are finite groups, and a morphism from $H$ to $G$ is an element in $R T^\Delta(kG,kH):=R\otimes_\ZZ T^\Delta(kG,kH)$, where $T^\Delta(kG,kH)$ is the Grothendieck group (with respect to split exact sequences) of the category of finitely generated $p$-permutation $(kG,kH)$-bimodules which are projective as left $kG$-modules and as right $kH$-modules. Composition of morphisms is induced by the tensor product of bimodules. We refer the reader to \cite{BoucYilmaz2020} and \cite{BoucYilmaz2022} for further notations and definitions related to this category. Diagonal $p$-permutation functors over $R$ form an abelian and $R$-linear category. \par
A {\em group-block pair} (over $k$) is by definition a pair $(G,b)$ of a finite group $G$ and a block idempotent $b$ of $kG$. A {\em defect group} of such a pair $(G,b)$ is a defect group of $b$ in $G$, and the fusion system of $(G,b)$ is the fusion system of $b$ in $kG$. To every group-block pair $(G,b)$ (over $k$), one can assign a diagonal $p$-permutation functor $R T^\Delta_{(G,b)}$ over $R$, see \cite{BoucYilmaz2022}. Two group-block pairs $(G,b)$ and $(H,c)$ are called {\em functorially equivalent} over $R$ if the functors $R T^\Delta_{(G,b)}$ and $R T^\Delta_{(H,c)}$ are isomorphic.

\smallskip
If $R=\FF$, it was shown in \cite{BoucYilmaz2022} that the category of diagonal $p$-permutation functors is semisimple, that two group-block pairs, which are functorially equivalent over $\FF$, have isomorphic defect groups, and that for a given finite $p$-group $D$ there are only finitely many functorial equivalence classes of group-block pairs with defect group isomorphic to~$D$.

\smallskip
In \cite{Yilmaz2024} it was shown that if $(G,b)$ and $(H,c)$ are group-block pairs with cyclic defect groups or $2$-blocks of defect $2$ or $3$, then $(G,b)$ and $(H,c)$ are functorially equivalent over~$\FF$ if and only if they have isomorphic fusion systems. The goal of this paper is to show a similar result for group-block pairs with dihedral, generalized quaternion and semidihedral defect groups.

\begin{theorem}\label{thm main}
Let $(G,b)$ and $(H,c)$ be group-block pairs with dihedral, generalized quaternion, or semidihedral defect groups. Then, $(G,b)$ and $(H,c)$ are functorially equivalent over $\FF$ if and only if they have isomorphic fusion systems.
\end{theorem}

Theorem~\ref{thm main} is proved by computing explicitly the multiplicities of the simple functors occurring as direct summands in the various block functors in each case, see Propositions~\ref{prop dihedral}, \ref{prop quaternion} and \ref{prop semidihedral}.  In this sense, we prove much more than what is necessary to prove Theorem~\ref{thm main}. We believe that these explicit formulas will be useful in the future. 

\smallskip
The saturated fusion systems for $2$-groups of dihedral, generalized quaternion, and semidihedral type are known and described in \cite[Example~I.3.8]{AKO2011}. Up to isomorphism, there exist precisely three such fusion systems on dihedral groups, two on the generalized quaternion group of order $8$, three on those of order $2^n\ge 16$, and four on semidihedral groups. All of them are realized by principal blocks of various groups. For details see Sections 4--6, where Theorem~\ref{thm main} is proved separately for these three types of groups. Together with the results in \cite{Yilmaz2024}, Theorem~\ref{thm main} implies the following corollary.

\begin{corollary}\label{cor main}
Let $(G,b)$ and $(H,c)$ be group-block pairs such that $kGb$ and $kHc$ have tame representation type. Then $(G,b)$ and $(H,c)$ are functorially equivalent over $\FF$ if and only if they have isomorphic fusion systems.
\end{corollary}

It is an open question if arbitrary group-block pairs that are functorially equivalent over $\FF$ must have isomorphic fusion systems. The converse is wrong in general, since there exist blocks with isomorphic fusion systems but different numbers of simple module, while the number of simple modules is determined by the functorial equivalence class over $\FF$, see \cite{Kiyota1984} or \cite[Example IV.5.42]{AKO2011}.

\smallskip
The paper is arranged as follows. In Section~2 we recall some relevant results on diagonal $p$-permutation functors and in Section~3 we prove some results on multiplicities of simple functors in block functors that are specific to the defect groups at hand. In Sections 4--6 we prove the main result for the three types of defect groups mentioned above.


\section{Preliminaries}

We recall some definitions and results from \cite{BoucYilmaz2022} that are relevant for this paper. Recall that $k$ is an algebraically closed field of characteristic $p>0$ and that $\FF$ is an algebraically closed field of characteristic zero.

\begin{nothing}\label{noth prel}
(a) A pair $(P,u)$ where $P$ is a $p$-group and $u\in\Aut(P)$ is an automorphism of $p'$-order (i.e., of order not divisible by $p$) is called a {\em $D^\Delta$-pair}. In this case we write $\Pu$ for the semidirect product $P\rtimes \langle u\rangle$. We say that two $D^\Delta$-pairs $(P,u)$ and $(Q,v)$ are {\em isomorphic} if there exists a group isomorphism $f\colon \Pu\to\Qv$ such that $f(u)$ is $\Qv$-conjugate to~$v$. We denote by $\Aut(P,u)$ the group of automorphisms of the $D^\Delta$-pair $(P,u)$ and by $\Out(P,u)$ the quotient $\Aut(P,u)/\Inn(\Pu)$, see \cite[Notation~6.8]{BoucYilmaz2022}.

\smallskip
(b) Let $G$ and $H$ be finite groups and let $b$ and $c$ be block idempotents of $kG$ and $kH$, respectively. We say that the pairs $(G,b)$ and $(H,c)$ are {\em functorially equivalent over $\FF$}, if the corresponding diagonal $p$-permutation functors $\FF T^\Delta_{(G,b)}$ and $\FF T^\Delta_{(H,c)}$ are isomorphic, see \cite[Definition~10.1]{BoucYilmaz2022}. Here, $\FF T^\Delta_{(G,b)}=\FF T^\Delta(-,kG)\circ{[kGb]}=\FF T^\Delta(-,kGb)$. By \cite[Lemma~10.2]{BoucYilmaz2022}, the pairs $(G,b)$ and $(H,c)$ are functorially equivalent over $\FF$ if and only if there exist $\omega\in\FF T^\Delta(kGb,kHc)$ and $\sigma\in\FF T^\Delta(kHc,kGb)$ such that 
\begin{equation*}
  \omega\cdot_H\sigma = [kGb]\quad\text{in $\FF T^\Delta(kGb,kGb)$}\quad\text{and}\quad  
  \sigma\cdot_G\omega=[kHc]\quad \text{in $\FF T^\Delta(kHc,kHc)$\,.}
\end{equation*}

\smallskip
(c) Recall from \cite[Corollary~6.15]{BoucYilmaz2022} that the category of diagonal $p$-permutation functors over $\FF$ is semisimple. Moreover, the simple diagonal $p$-permutation functors $S_{L,u,V}$, up to isomorphism, are parametrized by the isomorphism classes of triples $(L,u,V)$, where $(L,u)$ is a $D^\Delta$-pair and $V$ is a simple $\FF \Out(L,u)$-module, each up to isomorphism.

\smallskip
(d) Since the category of diagonal $p$-permutation functors over $\FF$ is semisimple, the functor $\FF T^\Delta_{(G,b)}$ is a direct sum of simple diagonal $p$-permutation functors $S_{L,u,V}$. Hence, two group-block pairs $(G,b)$ and $(H,c)$ are functorially equivalent over $\FF$ if and only if, for any triple $(L,u,V)$, the multiplicities of the simple functor $S_{L,u,V}$ in $\FF T^\Delta_{(G,b)}$ and $\FF T^\Delta_{(H,c)}$ are the same. In the next Theorem we recall a formula for these multiplicities and we will need to recall the following notation.

\smallskip
(e) Let $(G,b)$ be a group-block pair and let $(D,e_D)$ be a maximal $(G,b)$-Brauer pair (over $k$). For any subgroup $P\le D$, let $e_P$ be the unique block idempotent of $kC_G(P)$ with $(P,e_P)\le (D,e_D)$. Moreover, let $\calF$ be the fusion system of $(G,b)$ with respect to $(D,e_D)$ and let $[\calF]$ be a set of $\calF$-isomorphism classes of objects of $\calF$, i.e., subgroups of $D$.

Let $(L,u)$ be a $D^\Delta$-pair. For $P\le D$ we set $\calP_{(P,e_P)}(L,u)$ to be the set of group isomorphisms $\pi\colon L\to P$ satisfying $\pi\circ u\circ\pi^{-1}\in\Aut_{\calF}(P)$. Thus, $\calP_{(P,e_P)}(L,u)=\emptyset$ unless $P\cong L$. The set $\calP_{(P,e_P)}(L,u)$ is an $\big(N_G(P,e_P),\Aut(L,u)\big)$-biset  via
\begin{equation*}
   g\cdot \pi\cdot \varphi =  i_g\circ\pi\circ\varphi\,,
\end{equation*}
for $g\in N_G(P, e_P)$, $\pi\in\calP_{(P,e_P)}(L,u)$ and $\varphi\in\Aut(L,u)$, where $i_g$ denotes the conjugation automorphism $x\mapsto gxg^{-1}$.
We denote by $[\calP_{(P,e_P)}(L,u)]$ a set of $N_G(P,e_P)\times\Aut(L,u)$-orbits of $\calP_{(P,e_P)}(L,u)$.

For $\pi\in\calP_{(P,e_P)}(L,u)$, the stabilizer in $\Aut(L,u)$ of the $N_G(P,e_P)$-orbit of $\pi$ is denoted by $\Aut(L,u)_{\overline{(P,e_P,\pi)}}$. Thus,
\begin{equation*}
   \Aut(L,u)_{\overline{(P,e_P,\pi)}} = \{\varphi\in\Aut(L,u)\mid \pi\varphi\pi^{-1} \in \Aut_{\calF}(P)\}\,.
\end{equation*}

Finally, for $P\le D$, $(L,u)$, and $\pi\in\calP_{(P,e_P)}(L,u)$ as above, we denote by $\PIM\big(kC_G(P)e_P,u\big)$ the set of isomorphism classes of projective indecomposable $kC_G(P)e_P$-modules that are fixed under $\pi u \pi^{-1}$. Further, we denote by $\FF \PIM\big(kC_G(P)e_P,u\big)$ its $\FF$-linear span. Note that $\Aut(L,u)_{\overline{(P,e_P,\pi)}}$ acts on $\PIM\big(kC_G(P)e_P,u\big)$ via $U\cdot \varphi:=\lexp{g}{U}$, where $g\in N_G(P,e_P)$ such that $i_g\pi\varphi = \pi$.
\end{nothing}

\begin{theorem}(\cite[Theorem~8.22(b)]{BoucYilmaz2022})\label{thm mult formula}
Let $(G,b)$ be a block, $(L,u)$ a $D^\Delta$-pair and $V$ an irreducible $\FF\Out(L,u)$-module. The multiplicity of the simple diagonal $p$-permutation functor $S_{L,u,V}$ in the functor $\FF T^\Delta_{(G,b)}$ is equal to the $\FF$-dimension of
\begin{equation*}
  \bigoplus_{P\in[\calF]}\ \bigoplus_{\pi\in[\calP_{(P,e_P)}(L,u)]}  \,\FF\PIM\big(kC_G(P)e_P,u\big)\otimes_{\Aut(L,u)_{\overline{(P,e_P,\pi)}}} V\,.
\end{equation*}
\end{theorem}

Note that by Theorem~\ref{thm mult formula}, $\FF T^\Delta_{(G,b)}$ is isomorphic to a direct sum of only finitely many simple functors $S_{L,u,V}$. Moreover, if $S_{L,u,V}$ is isomorphic to a direct summand of $\FF T^\Delta_{(G,b)}$ then $L$ is isomorphic to a subgroup of a defect group of $(G,b)$. In addition to the notation introduced in \ref{noth prel}, we'll need the following.

\begin{notation}
Let $(G,b)$ be a group-block pair.

\smallskip
(a) For a $D^\Delta$-pair $(L,u)$ and an irreducible $\FF \Out(L,u)$-module, we denote by $m\big(S_{L,u,V},\FF T^\Delta_{(G,b)}\big)$ the multiplicity of $S_{L,u,V}$ as a direct summand of $\FF T^\Delta_{(G,b)}$.

\smallskip
(b) By $l(G,b)$, we denote the number of isomorphism classes of simple $kGb$-modules. By \cite[Corollary~8.23]{BoucYilmaz2022}, one has $l(G,b)= m\big(S_{1,1,\FF},\FF T^\Delta_{(G,b)}\big)$.

\smallskip
(c) Let $(D,e_D)$ be a maximal $(G,b)$-Brauer pair and $\calF$ the fusion system of $(G,b)$ with respect to $(D,e_D)$. Following \cite{AKO2011}, for any $P\le D$, we set $\Out_\calF(P):=\Aut_F(P)/\Inn(P)$ and $\Out_D(P) = \Aut_D(P)/\Inn(P)$. Note that $\Out_D(P)\le \Out_\calF(P)\le \Out(P)$.
\end{notation}


\section{More on multiplicities of simple functors in block functors}

The results in this section will be used to calculate the multiplicities of the simple functors in block functors in Sections~4--6.

\begin{remark}\label{rem simple functors}
We list the possible triples $(L,u,V)$ parametrizing simple diagonal $p$-permutation functors over $\FF$ for some $p$-groups $L$, cf.~\ref{noth prel}(a),(c).

\smallskip
(a) For $L=1$, the only possible triple is $(1,1,\FF)$.

\smallskip
(b) Let $L$ be a cyclic group of order $2^n$. Then $\Aut(L)$ is an abelian $2$-group. It follows that the possible triples $(L,u,V)$ are of the form $(L,1,V)$, where $V$ is an $\FF\Aut(L)$-module. Note that $V$ has $\FF$-dimension one.

\smallskip
(c) Let $L$ be a Klein-four group. Then $\Aut(L)=\Out(L)\cong S_3$. Up to isomorphism, the only $D^\Delta$-pairs with first entry $L$ are $(L,1)$ and $(L,u_0)$, where $u_0\in\Aut(L)$ is an element of order $3$. We have $\Out(L,1)\cong S_3$, $\Aut(L,u_0)\cong A_4$, and $\Out(L,u_0)\cong 1$. Hence, the only simple functors arising from $L=V_4$ are $S_{V_4,1,\FF}$, $S_{V_4,1,\FF_-}$, $S_{V_4,1,V_2}$, and $S_{V_4,u_0,\FF}$, where $\FF_-$ is the sign representation of $S_3$ and $V_2$ is an irreducible $\FF S_3$-module of dimension $2$.

\smallskip
(d) Let $L=Q_8$ be the quaternion group of order $8$. Then $\Aut(L)\cong S_4$ and $\Out(L)\cong S_3$. Let $u_0\in\Aut(L)$ be an element of order $3$. Up to isomorphism, the only $D^\Delta$-pairs with first entry $L$ are $(L,1)$ and $(L,u_0)$. We have $\Out(L,1)=\Out(L)\cong S_3$ and $\Out(L,u_0)=1$. Thus, the simple functors arising from $L=Q_8$ are $S_{Q_8,1,\FF}$, $S_{Q_8,1,\FF_-}$, $S_{Q_8, 1,V_2}$, and $S_{Q_8, u_0,\FF}$, where $\FF_-$ and $V_2$ are as in Part~(c).

\smallskip
(e) Let $L$ be a dihedral group of order $2^n\ge 8$ or generalized quaternion or semidihedral of order $2^n\ge 16$. Then $\Out(L)$ is an abelian $2$-group determined in more detail in Sections~4--6. Thus, the simple functors arising from $L$ are all of the form $S_{L,1,V}$, where $V$ is a one-dimensional $\FF\Out(L)$-module.
\end{remark}

\begin{remark}\label{rem iso}
For the following it is useful to be aware of the following fact. Let $f\colon G\to H$ be an isomorphism of groups. Then $f$ induces isomorphisms $\Aut(f)\colon\Aut(G)\to\Aut(H)$ and $\Out(f)\colon\Out(G)\to\Out(H)$ in the obvious way, which of course depend on $f$. In contrast, the isomorphisms between the module categories $\lMod{\Aut(H)}\to\lMod{\Aut(G)}$ and $\lMod{\Out(H)}\to \lMod{\Out(G)}$, induced by restriction along the isomorphisms $\Aut(f)$ and $\Out(f)$, respectively, do not depend on the choice of $f$, up to natural isomorphism.
\end{remark}

The following lemmas will be used in Sections~4--6.

\begin{lemma}(\cite[Lemma~2.4]{Yilmaz2024})\label{lem 1}
Let $(G,b)$ be a group-block pair, let $(D,e_D)$ be a maximal $(G,b)$-Brauer pair and let $\calF$ denote the fusion system of $(G,b)$ with respect to $(D,e_D)$. Then, for any simple $\FF \Out(D)$-module $V$, one has
\begin{equation*}
   m\big(S_{D,1,V},\FF T^\Delta_{(G,b)}\big) = \dim_\FF\bigl(V^{\Out_\calF(D)}\bigr)\,,
\end{equation*}
where $V^{\Out_\calF(D)}$ denotes the $\Out_\calF(D)$-fixed points of $V$.
\end{lemma}

The following lemma is a generalization of Lemma~\ref{lem 1}

\begin{lemma}\label{lem 2}
Let $(G,b)$ be a group-block pair, let $(D,e_D)$ be a maximal $(G,b)$-Brauer pair and let $\calF$ denote the fusion system of $(G,b)$ with respect to $(D,e_D)$.
Further, suppose that $L$ is a $p$-group and that for every subgroup $P\le D$ with $P\cong L$ one has $l\big(kC_G(P)e_P\big)=1$. Then, for every simple $\FF\Out(L)$-module $V$, one has
\begin{equation*}
   m\big(S_{L,1,V}, \FF T^\Delta_{(G,b)}\big) = \sum_{L\cong P\in[\calF]} \dim_\FF\bigl(V^{\Out_\calF(P)}\bigr)\,,
\end{equation*}
where $V$ is viewed as an $\Out(P)$-module via any isomorphism $L\cong P$, cf.~Remark~\ref{rem iso}.
\end{lemma}

\begin{proof}
The proof is similar to the proof of Lemma~\ref{lem 1} given in \cite{BoucYilmaz2022}. We add the details for convenience. Let $P\le D$. Note that $\calP_{(P,e_P)}(L,1)$ is the set of all isomorphisms from $L$ to $P$. Suppose that $L\cong P$ and let $\pi\colon L\to P$ an isomorphism. Then $\Aut(L,1)$ acts transitively on $\calP_{(P,e_P)}(L,1)$ and we may choose $[\calP_{(P,e_P)}(L,1)]=\{\pi\}$. With this, we have
\begin{equation*}
   \Aut(L,1)_{\overline{(P,e_P,\pi)}} = \{\varphi\in\Aut(L)\mid \pi\varphi\pi^{-1}\in\Aut_\calF(P)\} = \pi^{-1}\Aut_\calF(P)\pi\,.
\end{equation*}
Since $l\big(kC_G(P)e_P \big)=1$, the formula in Theorem~\ref{thm mult formula} implies that
\begin{align*}
   m\big(S_{L,1,V},\FF T^\Delta_{(G,b)}\big) &= \sum_{L\cong P\in[\calF]} \dim_\FF\bigl(\FF\otimes_{\FF[\pi^{-1}\Aut_\calF(P)\pi]} V\bigr)\\
   & = \sum_{L\cong P\in[\calF]} \dim_\FF\bigl(V^{\pi^{-1}\Aut_\calF(P)\pi}\bigr)
   = \sum_{L\cong P\in[\calF]} \dim_\FF\bigl(V^{\Out_\calF(P)}\bigr)\,.
\end{align*}
\end{proof}

\begin{lemma}\label{lem 3}
Let $(G,b)$ be a group-block pair, let $(D,e_D)$ be a maximal $(G,b)$-Brauer pair and let $\calF$ denote the fusion system of $(G,b)$ with respect to $(D,e_D)$. Furthermore, let $L$ be a Klein-four group, let $(L,u)$ be a $D^\Delta$-pair and $V$ a simple $\FF\Out(L,u)$-module. Finally, suppose that $L\cong P\le D$. 

\smallskip
{\rm (a)} If $\Out_\calF(P)\cong C_2$ then (cf.~Remark~\ref{rem simple functors}(c))
\begin{equation*}
   \dim_\FF\Biggl(\bigoplus_{\pi\in[\calP_{(P,e_P})(L,u)]} \FF\otimes_{\Aut(L,u)_{\overline{(P,e_P,\pi)}}} V \Biggr) =
   \begin{cases}
      1\,, & \text{if $(u,V) = (1,\FF)$;}\\
      0\,, & \text{if $(u,V) = (1,\FF_-)$;}\\
      1\,, & \text{if $(u,V) = (1, V_2)$;}\\
      0\,, & \text{if $(u,V) = (u_0,\FF)$.}
   \end{cases}
\end{equation*}

\smallskip
{\rm (b)} If $\Out_{\calF}(P)\cong S_3$ then (cf.~Remark~\ref{rem simple functors}(c))
\begin{equation*}
   \dim_\FF\Biggl(\bigoplus_{\pi\in[\calP_{(P,e_P})(L,u)]} \FF\otimes_{\Aut(L,u)_{\overline{(P,e_P,\pi)}}} V \Biggr) =
   \begin{cases}
      1\,, & \text{if $(u,V) = (1,\FF)$;}\\
      0\,, & \text{if $(u,V) = (1,\FF_-)$;}\\
      0\,, & \text{if $(u,V) = (1, V_2)$;}\\
      1\,, & \text{if $(u,V) = (u_0,\FF)$.}
   \end{cases}
\end{equation*}   
\end{lemma}

\begin{proof}
We may assume that $P=L$ and make use of Remark~\ref{rem simple functors}(c).

\smallskip
(a) Suppose that $\Out_\calF(P)\cong C_2$. If $u=1$ then
\begin{equation*}
   \calP_{(P,e_P)}(L,u) = \Aut(L)\cong S_3\quad\text{and} \quad [\calP_{(P,e_P)}(L,u)]=\{\id\}\,.
\end{equation*}
This implies that
\begin{equation*}
   \Aut(L,1)_{\overline{(P,e_P,\id)}} = \{\varphi\in\Aut(L)\mid \varphi\in\Aut_\calF(P)\} \cong C_2\,,
\end{equation*}
and hence that
\begin{equation*}
   \dim_\FF\Biggl(\bigoplus_{\pi\in[\calP_{(P,e_P})(L,u)]} \FF\otimes_{\Aut(L,u)_{\overline{(P,e_P,\pi)}}} V \Biggr) =
   \dim_\FF (\FF\otimes _{\FF C_2} V) = \dim_\FF(V^{C_2})\,.
\end{equation*}
The latter is equal to $1$ if $V=\FF$ or $V\cong V_2$ and equal to $0$ otherwise. For $u=u_0$ of order~$3$, we have
\begin{equation*}
   \calP_{(P,e_P)}(L,u) = \{\varphi\in\Aut(L)\mid \varphi u_0 \varphi^{-1}\in\Aut_\calF(P)\} = \emptyset\,,
\end{equation*}
and the result follows.

\smallskip
(b) Suppose that $\Out_\calF(P)\cong S_3$ and let $u=1$. Then we have again
\begin{equation*}
   \calP_{(P,e_P)}(L,u) = \Aut(L)\cong S_3\quad\text{and}\quad [\calP_{(P,e_P)}(L,u)] = \{\id\},.
\end{equation*}
This implies that
\begin{equation*}
   \Aut(L,u)_{\overline{(P,e_P,\id)}} = \{\varphi\in\Aut(L)\mid \varphi\in\Aut_\calF(P)\} \cong S_3\,,
\end{equation*}
and hence that
\begin{equation*}
   \dim_\FF\Biggl(\bigoplus_{\pi\in[\calP_{(P,e_P})(L,u)]} \FF\otimes_{\Aut(L,u)_{\overline{(P,e_P,\pi)}}} V \Biggr) =
   \dim_\FF (\FF\otimes _{\FF S_3} V) = \dim_\FF(V^{S_3})\,.
\end{equation*}
The latter is nonzero only if $V=\FF$ and the result follows if $u=1$. If $u=u_0$ has order $3$ then 
\begin{equation*}
   \calP_{(P,e_P)}(L,u) = \{\varphi\in\Aut(L)\mid \varphi u_0 \varphi^{-1}\in\Aut_\calF(P)\} = \Aut(L)\cong S_3\,,
\end{equation*}
and we may choose $[\calP_{(P,e_P)}(L,u)]=\{\id\}$. Since $\Out(V_4,u_0)=1$, we obtain
\begin{equation*}
   \dim_\FF\Biggl(\bigoplus_{\pi\in[\calP_{(P,e_P})(L,u)]} \FF\otimes_{\Aut(L,u)_{\overline{(P,e_P,\pi)}}} V \Biggr) = \dim_\FF(\FF\otimes_\FF \FF) = 1
\end{equation*}
as desired.
\end{proof}

The proof of the following lemma is similar to the proof of Lemma~\ref{lem 2} and is left to the reader.

\begin{lemma}\label{lem 4}
Let $(G,b)$ be a group-block pair, let $(D,e_D)$ be a maximal $(G,b)$-Brauer pair and let $\calF$ denote the fusion system of $(G,b)$ with respect to $(D,e_D)$. Furthermore, let $L$ be a quaternion group of order $8$, let $(L,u)$ be a $D^\Delta$-pair and $V$ a simple $\FF\Out(L,u)$-module. Finally, suppose that $L\cong P\le D$. 

\smallskip
{\rm (a)} If $\Out_\calF(P)\cong C_2$ then (cf.~Remark~\ref{rem simple functors}(d))
\begin{equation*}
   \dim_\FF\Biggl(\bigoplus_{\pi\in[\calP_{(P,e_P})(L,u)]} \FF\otimes_{\Aut(L,u)_{\overline{(P,e_P,\pi)}}} V \Biggr) =
   \begin{cases}
      1\,, & \text{if $(u,V) = (1,\FF)$;}\\
      0\,, & \text{if $(u,V) = (1,\FF_-)$;}\\
      1\,, & \text{if $(u,V) = (1, V_2)$;}\\
      0\,, & \text{if $(u,V) = (u_0,\FF)$.}
   \end{cases}
\end{equation*}

\smallskip
{\rm (b)} If $\Out_{\calF}(P)\cong S_3$ then (cf.~Remark~\ref{rem simple functors}(c))
\begin{equation*}
   \dim_\FF\Biggl(\bigoplus_{\pi\in[\calP_{(P,e_P})(L,u)]} \FF\otimes_{\Aut(L,u)_{\overline{(P,e_P,\pi)}}} V \Biggr) =
   \begin{cases}
      1\,, & \text{if $(u,V) = (1,\FF)$;}\\
      0\,, & \text{if $(u,V) = (1,\FF_-)$;}\\
      0\,, & \text{if $(u,V) = (1, V_2)$;}\\
      1\,, & \text{if $(u,V) = (u_0,\FF)$.}
   \end{cases}
\end{equation*}   
\end{lemma}


\section{Blocks with dihedral defect groups}

Throughout this section we assume that $(G,b)$ is a group-block pair with defect group $D$ isomorphic to $D_{2^n}=\langle s,t\mid s^{2^{n-1}}=t^2=1,tst=s^{-1}\rangle$, the dihedral group of order $2^n$, with $n\ge 3$. First we recall a result that holds for dihedral groups of arbitrary order.

\begin{remark}\label{rem Aut dihedral}
Let $m\ge 3$ be an integer and let $D_{2m}=\langle s,t\mid s^m=t^2=1, tst=s^{-1}\rangle$ be a dihedral group of order $2m$. It is well-known and easy to see that the automorphisms of $D_{2m}$ are of the form $f_{a,b}$, where $f_{a,b}(s)=s^b$ and $f_{a,b}(t)=ts^a$, with integers $a,b$ such that $\gcd(m,b)=1$. It is easily verified that the map
\begin{equation*}
   \ZZ/m\ZZ\rtimes (\ZZ/m\ZZ)^\times \to \Aut(D_{2m})\,,\quad (\abar,\bbar)\mapsto f_{a,b}\,,
\end{equation*}
is an isomorphism, where the semidirect product is formed with respect to the action of $\bbar$ on $\abar$ by multiplication in the ring $\ZZ/m\ZZ$. 
Moreover, the inner automorphism given by conjugation with $t^ls^k$ is equal to $f_{-2k,(-1)^l}$, $k,l\in\ZZ$, so that under the above isomorphism, $\Inn(D_{2^n})$ corresponds to the subgroup $2(\ZZ/m\ZZ)\rtimes\{\pm 1+m\ZZ\}$. 
Thus, the above isomorphism induces an isomorphism
\begin{equation*}
   \Out(D_{2m})\cong 
   \begin{cases}
      \ZZ/2\ZZ\times \big((\ZZ/m\ZZ)^\times/\{\pm 1+m\ZZ\}\big)\,, & \text{if $m$ is even;}\\
      (\ZZ/m\ZZ)^\times /\{\pm 1+m\ZZ\}\,, & \text{ if $m$ is odd.}
   \end{cases}
\end{equation*}
In particular, $\Out(D_{2m})$ is abelian and 
\begin{equation*}
   \Out(D_{2^n}) \cong \ZZ/2\ZZ\times \big((\ZZ/2^{n-1}\ZZ)^\times/\{\pm 1+2^{n-1}\ZZ\}\big)
\end{equation*}
is an abelian $2$-group.
\end{remark}

\begin{nothing}\label{noth dihedral}
(a) Following \cite[Section~1]{Brauer1974}, we define the following subgroups of $D_{2^n}$:
\begin{equation*}
   S_m:=\langle s_m\rangle \quad\text{with $s_m:=s^{2^{n-m-1}}$,}
\end{equation*}
a cyclic subgroup of order $2^m$, for $0\le m\le n-1$, and
\begin{equation*}
   W^1_m:=\langle s_{m-1},t\rangle\quad\text{and}\quad W^2_m:=\langle s_{m-1}, st\rangle\quad \text{for $1\le m\le n-1$.}
\end{equation*}
Then the subgroups $D_{2^n}$, $S_m$ ($0\le m\le n-1$) and $W^i_m$ ($i\in\{1,2\}$ and $1\le m\le n-1$), form a complete set of representatives of the conjugacy classes of subgroups of $D_{2^n}$. Moreover, $W^i_m$ (for $i=1,2$ and $3\le m\le n-1$) is isomorphic to a dihedral group of order $2^m$ and $W^i_2$, for $i=1,2$ are isomorphic to Klein-four groups.  For $i=1,2$ and $2\le m\le n-2$, one has $N_{D_{2^n}}(W_m^i)=W_{m+1}^i$.

\smallskip
(b) Up to isomorphism, there exist three saturated fusion systems on $D$ (see \cite[Example~3.8]{AKO2011} and \cite[Theorem~5.3]{CravenGlesser2012}). Following \cite{AKO2011} we denote by $\calF_{00}$ (resp.~$\calF_{01}$, $\calF_{11}$) the fusion system with three (resp.~two, one) isomorphism classes of involutions. They can be realized, respectively, as the fusion systems of the groups $D$, of the principal blocks of $\PGL(2,q)$ for any odd prime power $q$ such that $2(q+1)_2=2^n$ or $2(q-1)_2=2^n$, and of the principal block of $\PSL(2,q)$ for any odd prime power $q$ such that $(q+1)_2=2^n$ or $(q-1)_2=2^n$. One has
\begin{equation*}
   \Aut_{\calF_{00}}(W_2^i)\cong C_2\,,\ \Aut_{\calF_{01}}(W_2^1)\cong C_2\,,\ \Aut_{\calF_{01}}(W_2^2)\cong S_3\,,\text{ and } 
   \Aut_{\calF_{11}}(W_2^i)\cong S_3\,,
\end{equation*}
for $i=1,2$. For each of the three fusion systems, the essential subgroups of $D$ are precisely the subgroups isomorphic to Klein-four groups.
\end{nothing}

The following theorem generalizes the result in \cite{Yilmaz2024} from $n=3$ to $n\ge 3$. 

\begin{proposition}\label{prop dihedral}
Let $(G,b)$ be a group-block pair, let $(D,e_D)$ be a maximal $(G,b)$-Brauer pair and let $\calF$ denote the fusion system of $(G,b)$ with respect to $(D,e_D)$.
Suppose that the defect group $D$ of $(G,b)$ is isomorphic to $D_{2^n}$, the dihedral group of order $2^n$.

\smallskip
{\rm (a)} Let $L=D_{2^n}$. Then, for any simple $\FF\Out(L)$-module $V$, one has $m\big(S_{L,1,V},\FF T^\Delta_{(G,b)}\big) = 1$.

\smallskip
{\rm (b)} Let $L=D_{2^m}$ with $3\le m\le n-1$ and let $\epsilon\in\Out(L)$ be the class of the automorphism of $L$ that is the identity on the cyclic subgroup of order $2^{m-1}$ and switches the two conjugacy classes of non-central involutions of $L$. Then, for any simple $\FF\Out(L)$-module $V$, one has
\begin{equation*}
   m\big(S_{L,1,V},\FF T^\Delta_{(G,b)}\big) = 
   \begin{cases}
      2\,, & \text{if $\epsilon\in\ker(V)$;}\\
      0\,, & \text{if $\epsilon\notin\ker(V)$.}
   \end{cases}
\end{equation*}

\smallskip
{\rm (c)} Let $L=C_{2^m}$ be a cyclic group of order $2^m$ with $2\le m\le n-1$, and let $V$ be a simple $\FF\Out(L)$-module. Moreover, let $\epsilon\in\Out(L)$ be the class of the inversion automorphism $x\mapsto x^{-1}$ of $L$. Then
\begin{equation*}
   m\big(S_{L,1,V},\FF T^\Delta_{(G,b)}\big) =
   \begin{cases}
      1\,, & \text{if $\epsilon\in\ker(V)$;}\\
      0\,, & \text{if $\epsilon\notin\ker(V)$.}
   \end{cases}
\end{equation*}

\smallskip
{\rm (d)} Let $L=V_4$ and let $a_1$, $a_2$, $a_3$, $a_4$ denote the multiplicities of $S_{V_4, 1,\FF}$, $S_{V_4,1,\FF_-}$, $S_{V_4,1,V_2}$, $S_{V_4,u_0,\FF}$ in $\FF T^\Delta_{(G,b)}$, respectively (see Remark~\ref{rem simple functors}(c)). Then
\begin{equation*}
   (a_1,a_2,a_3,a_4)=
   \begin{cases}
      (2,0,2,0)\,, & \text{if $\calF\cong \calF_{00}$;}\\
      (2,0,1,1)\,, & \text{if $\calF\cong \calF_{01}$;}\\
      (2,0,0,2)\,, & \text{if $\calF\cong \calF_{11}$.}\\
   \end{cases}
\end{equation*}

\smallskip
{\rm (e)} Let $L=C_2$. Then
\begin{equation*}
   m\big(S_{C_2,1,\FF}, \FF T^\Delta_{(G,b)}\big)=
   \begin{cases}
      3\,, & \text{if $\calF\cong \calF_{00}$;}\\
      2\,, & \text{if $\calF\cong \calF_{01}$;}\\
      1\,, & \text{if $\calF\cong \calF_{11}$.}\\
   \end{cases}
\end{equation*}
   
\smallskip
{\rm (f)} Let $L=\{1\}$. Then
\begin{equation*}
   m\big(S_{\{1\},1,\FF},\FF T^\Delta_{(G,b)}\big)=
   \begin{cases}
      1\,, & \text{if $\calF\cong \calF_{00}$;}\\
      2\,, & \text{if $\calF\cong \calF_{01}$;}\\
      3\,, & \text{if $\calF\cong \calF_{11}$.}\\
   \end{cases}
\end{equation*}
\end{proposition}

\begin{proof}
(a) Since $\Aut(D)$ is a $2$-group (see Remark~\ref{rem Aut dihedral}), the Sylow axiom for saturated fusion systems implies that $\Out_\calF(D)= 1$. Moreover, $V$ is $1$-dimensional, since $\Out(D)$ is an abelian group by Remark~\ref{rem Aut dihedral}. The result follows now from Lemma~\ref{lem 1}.

\smallskip
(b) We will use Lemma~\ref{lem 2}. We first note that there are precisely two $\calF$-isomorphism classes of subgroups of $D$ isomorphic to $D_{2^m}$. In fact there are two conjugacy classes of subgroups of $D$ isomorphic to $D_{2^m}$. 
Moreover, they are not fused in $\calF$, since the Klein-four groups are the only $\calF$-essential subgroups of $D$, and if they were fused they would have to be conjugate by an element in $D$, by Alperin's fusion theorem. Now let $D_{2^m}\cong P\le D$. 
Again, by Alperin's fusion theorem, $\Aut_{\calF}(P)$ is given by conjugations with elements from $N_D(P)$. But $N_D(P)$ is isomorphic to a dihedral subgroup of $D$ of order $2^{m+1}$. Thus, $\Out_{\calF}(P)\cong C_2$. The non-trivial element of $\Out_{\calF}(P)$ is the class of the conjugation automorphism with any element of $N_D(P)$ which is not in $P$. We may choose this element to be an element of order $2^m$. 
Note that this element centralizes the cyclic subgroup of order $2^{m-1}$ of $P$. Therefore, under any isomorphism $P\cong L$ the corresponding element of $L$ will again centralize the cyclic subgroup of $L$ of order $2^{m-1}$. It follows that the class of the conjugation with this element is equal to $\epsilon$ as described in the result. 
Finally, by choosing $P$ to be fully $\calF$-centralized, we can see that $C_D(P)=Z(D)$ is a defect group of the group-block pair $(C_G(P), e_P)$. Since $Z(D)=Z(P)$ is central in $C_G(P)$, we obtain $l\big(kC_G(P)e_P\big)=1$. The result now follows from Lemma~\ref{lem 2}, and the fact that $V$ is one-dimensional, since $\Out(L)$ is abelian.

\smallskip
(c) We will use again Lemma~\ref{lem 2}. First note that there is a unique subgroup $P$ of $D$ isomorphic to $L$. Therefore $P$ is fully $\calF$-centralized and $C_D(P)$ is a defect group of $(C_G(P), e_P)$. But $C_D(P)$ is the cyclic subgroup of order $2^{n-1}$ of $D$. Since it is cyclic and $p=2$, the block $kC_G(P)e_P$ is nilpotent and we obtain $l\big(kC_G(P)e_P\big)=1$. Moreover, by Alperin's fusion theorem (the only $\calF$-essential subgroups of $D$ are isomorphic to Klein-four groups), we obtain $\Out_\calF(P)\cong N_D(P)/PC_D(P)\cong C_2$ and that the non-trivial element of $\Out_\calF(P)$ is given by the inversion automorphism of $P$. The result now follows from Lemma~\ref{lem 2}.

\smallskip
(d) We will use Lemma~\ref{lem 3}. By Alperin's fusion theorem, the two conjugacy classes of subgroups of $D$ isomorphic to $V_4$ are also the $\calF$-isomorphism classes. We may choose $V_4\cong P\le D$ to be fully $\calF$-centralized. Then $C_D(P)=P$ is a defect group of $(C_G(P), e_P)$ and it is central in $C_G(P)$. Thus, $l\big(kC_G(P)e_P\big)=1$.  The result now follows from  Theorem~\ref{thm mult formula} and Lemma~\ref{lem 3}.

\smallskip
(e) For any subgroup $P$ of $D$ of order $2$, one has $l\big(kC_G(P)e_P\big)=1$ by the last paragraph before Proposition~4B and by Proposition~4F of \cite{Brauer1974}. Lemma~\ref{lem 2} implies that $m\big(S_{C_2,1,\FF},\FF T^\Delta_{(G,b)}\big)$ is equal to the number of $\calF$-isomorphism classes of subgroups of $D$ of order $2$. The result follows.

\smallskip
(f) Since $m\big(S_{1,1,\FF},\FF T^\Delta_{(G,b)}\big) = l(kGb)$, see \cite[Corollary~8.23]{BoucYilmaz2022}, the result follows immediately from \cite[Theorem~2]{Brauer1974}.
\end{proof}

If a simple functor $S_{L,u,V}$ occurs as a direct summand of $\FF T^\Delta_{(G,b)}$ then $L$ is isomorphic to a subgroup of the defect group $D$ of $(G, b)$. Therefore, Proposition~\ref{prop dihedral} covers all possible simple functors. Since the multiplicities computed in Proposition~\ref{prop dihedral} only depend on the fusion system of $(G,b)$ and since non-isomorphic fusion systems produce different multiplicities of $S_{1,1,\FF}$, this proves the part of Theorem~\ref{thm main} concerning dihedral defect groups.


\section{Blocks with generalized quaternion defect groups}

Throughout this section we assume that $(G,b)$ is a group-block pair with defect group $D$ isomorphic to $Q_{2^n}=\langle s,t\mid s^{2^{n-1}}=t^4=1,t^{-1}st=s^{-1}, s^{2^{n-2}}=t^2\rangle$, the generalized quaternion group of order $2^n$, with $n\ge 4$. The case $n=3$ has been already established in~\cite{Yilmaz2024}.

\begin{remark}\label{rem Aut quaternion}
It is well-known and an easy verification that every automorphism of $Q_{2^n}$ (for $n\ge 4$) is of the form $f_{a,b}$, where $f_{a,b}(s)=s^b$ and $f_{a,b}(t)= ts^a$ for some $a,b\in\ZZ$ with $\gcd(2,b)=1$. This induces a group isomorphism
\begin{equation*}
   \ZZ/2^{n-1}\ZZ \rtimes (\ZZ/2^{n-1}\ZZ)^\times \to \Aut(Q_{2^n})\,,\quad (\abar,\bbar)\mapsto f_{a,b}\,,
\end{equation*}
where in the semidirect product the multiplicative group $(\ZZ/2^{n-1}\ZZ)^\times$ acts by multiplication on the additive group $\ZZ/2^{n-1}\ZZ$.
It is easy to see that conjugation with an element $t^ls^k$, for integers $k,l$, is equal to $f_{-2k, (-1)^l}$. This implies that under the above isomorphism $\Inn(Q_{2^n})$ corresponds to the subgroup $2\ZZ/2^{n-1}\ZZ \rtimes \{\pm 1+2^{n-1}\ZZ\}$ and that the above isomorphism induces an isomorphism
\begin{equation*}
  \Out(Q_{2^n})\cong  \ZZ/2\ZZ\times \big((\ZZ/2^{n-1}\ZZ)^\times/\{\pm 1+2^{n-1}\ZZ\}\big)\,.
\end{equation*}
In particular, $\Aut(Q_{2^n})$ is a $2$-group and $\Out(Q_{2^n})$ is an abelian $2$-group.
\end{remark}

\begin{nothing}\label{noth quaternion}
(a) Similar to \ref{noth dihedral}, we define the following subgroups of $Q_{2^n}$:
\begin{equation*}
   S_m:=\langle s_m\rangle \quad\text{with $s_m:=s^{2^{n-m-1}}$,}
\end{equation*}
a cyclic subgroup of order $2^m$ for $0\le m\le n-1$, and
\begin{equation*}
   T^1_m:=\langle s_{m-1},t\rangle\quad\text{and}\quad T^2_m:=\langle s_{m-1}, st\rangle\quad \text{for $2\le m\le n-1$.}
\end{equation*}
Then the subgroups $Q_{2^n}$, $S_m$ ($0\le m\le n-1$) and $T^i_m$ ($i\in\{1,2\}$ and $2\le m\le n-1$), form a complete set of representatives of the conjugacy classes of subgroups of $Q_{2^n}$. Moreover, $T^i_m$ (for $i=1,2$ and $4\le m\le n-1$) is isomorphic to a generalized quaternion group of order $2^{m}$ and $T^i_3$, for $i=1,2$ are isomorphic to quaternion groups of order $8$. For $i=1,2$ and $2\le m\le n-2$, one has $N_{Q_{2^n}}(T_m^i)=T_{m+1}^i$.

\smallskip
(b) Up to isomorphism, there exist three saturated fusion systems on $D$ (see \cite[Example~3.8]{AKO2011} and \cite[Theorem~5.3]{CravenGlesser2012}). Following again \cite{AKO2011} we denote by $\calF_{00}$ (resp.~$\calF_{01}$, $\calF_{11}$) the fusion system with three (resp.~two, one) isomorphism classes of cyclic groups of order $4$. They can be realized, respectively, as the fusion systems of the groups $D$, of the principal blocks of $2.\PGL(2,q)$ for any odd prime power $q$ such that $2(q+1)_2=2^{n-1}$ or $2(q-1)_2=2^n$, and of the principal block of $\SL(2,q)$ for any odd prime power $q$ such that $(q+1)_2=2^{n-1}$ or $(q-1)_2=2^{n-1}$. One has
\begin{equation*}
   \Aut_{\calF_{00}}(T_3^i)\cong C_2\,,\ \Aut_{\calF_{01}}(T_3^1)\cong C_2\,,\ \Aut_{\calF_{01}}(T_3^2)\cong S_3\,,\text{ and } 
   \Aut_{\calF_{11}}(T_2^i)\cong S_3\,,
\end{equation*}
for $i=1,2$. For each of the three fusion systems, the essential subgroups of $D$ are precisely the subgroups isomorphic to quaternion groups of order $8$.
\end{nothing}

\begin{proposition}\label{prop quaternion}
Let $(G,b)$ be a group-block pair, let $(D,e_D)$ a maximal $(G,b)$-Brauer pair and let $\calF$ denote the fusion system of $(G,b)$ with respect to $(D,e_D)$.
Suppose that the defect group $D$ of $(G,b)$ is isomorphic to $Q_{2^n}$, the generalized quaternion group of order~$2^n$ with $n\ge 4$.

\smallskip
{\rm (a)} Let $L=Q_{2^n}$. Then, for any simple $\FF\Out(L)$-module $V$, one has $m\big(S_{L,1,V},\FF T^\Delta_{(G,b)}\big) = 1$.

\smallskip
{\rm (b)} Let $L=Q_{2^m}$ with $4\le m\le n-1$ and let $\epsilon\in\Out(L)$ be the class of the automorphism of $L$ that is the identity on the cyclic subgroup of order $2^{m-1}$ and switches the two conjugacy classes of subgroups of order $4$  of $L$ not contained in the cyclic subgroup of order $2^{m-1}$. Then, for any simple $\FF\Out(L)$-module $V$, one has
\begin{equation*}
   m\big(S_{L,1,V},\FF T^\Delta_{(G,b)}\big) = 
   \begin{cases}
      2\,, & \text{if $\epsilon\in\ker(V)$;}\\
      0\,, & \text{if $\epsilon\notin\ker(V)$.}
   \end{cases}
\end{equation*}

\smallskip
{\rm (c)} Let $L=C_{2^m}$ be a cyclic group of order $2^m$ with $3\le m\le n-1$, and let $V$ be a simple $\FF\Out(L)$-module. Moreover, let $\epsilon\in\Out(L)$ be the class of the inversion automorphism $x\mapsto x^{-1}$ of $L$. Then
\begin{equation*}
   m\big(S_{L,1,V},\FF T^\Delta_{(G,b)}\big) =
   \begin{cases}
      1\,, & \text{if $\epsilon\in\ker(V)$;}\\
      0\,, & \text{if $\epsilon\notin\ker(V)$.}
   \end{cases}
\end{equation*}

\smallskip
{\rm (d)} Let $L=Q_8$ and let $a_1$, $a_2$, $a_3$, $a_4$ denote the multiplicities of $S_{Q_8, 1,\FF}$, $S_{Q_8,1,\FF_-}$, $S_{Q_8,1,V_2}$, $S_{Q_8,u_0,\FF}$ in $\FF T^\Delta_{(G,b)}$, respectively (see Remark~\ref{rem simple functors}(d)). Then
\begin{equation*}
   (a_1,a_2,a_3,a_4)=
   \begin{cases}
      (2,0,2,0)\,, & \text{if $\calF\cong \calF_{00}$;}\\
      (2,0,1,1)\,, & \text{if $\calF\cong \calF_{01}$;}\\
      (2,0,0,2)\,, & \text{if $\calF\cong \calF_{11}$.}\\
   \end{cases}
\end{equation*}

\smallskip
{\rm (e)} Let $L=C_4$. Then
\begin{equation*}
m\big(S_{C_4,1,\FF_-}, \FF T^\Delta_{(G,b)}\big)=0\quad \text{and}\quad  m\big(S_{C_4,1,\FF}, \FF T^\Delta_{(G,b)}\big)=
   \begin{cases}
      3\,, & \text{if $\calF\cong \calF_{00}$;}\\
      2\,, & \text{if $\calF\cong \calF_{01}$;}\\
      1\,, & \text{if $\calF\cong \calF_{11}$.}\\
   \end{cases}
\end{equation*}

\smallskip
{\rm (f)} One has
\begin{equation*}
   m\big(S_{1,1,\FF},\FF T^\Delta_{(G,b)}\big)= m\big(S_{C_2,1,\FF},\FF T^\Delta_{(G,b)}\big)=
   \begin{cases}
      1\,, & \text{if $\calF\cong \calF_{00}$;}\\
      2\,, & \text{if $\calF\cong \calF_{01}$;}\\
      3\,, & \text{if $\calF\cong \calF_{11}$.}\\
   \end{cases}
\end{equation*}
\end{proposition}
\begin{proof}
(a) Since the automorphism group $\Aut(D)$ is a $2$-group (see Remark~\ref{rem Aut quaternion}),  the Sylow axiom for saturated fusion systems implies that $\Out_\calF(D)=1$. Moreover, since $\Out(D)$ is an abelian group (again by Remark~\ref{rem Aut quaternion}), the module $V$ has dimension one. The result follows now from Lemma~\ref{lem 1}.

\smallskip
(b) We will use Lemma~\ref{lem 2}. By arguments similar to the proof of Proposition~\ref{prop dihedral}{\rm (b)}, one shows that there are precisely two $\calF$-isomorphism classes of subgroups of $D$ isomorphic to $Q_{2^m}$. Let $Q_{2^m}\cong P\le D$. Since the quaternion subgroups of order $8$ are the only $\calF$-essential subgroups of $D$, Alperin's fusion theorem implies that $\Aut_\calF(P)$ is given by conjugations with elements from $N_D(P)$. But $N_D(P)$ is isomorphic to a generalized quaternion subgroup of $D$ of order $2^{m+1}$. Thus, $\Out_{\calF}(P)\cong C_2$. The non-trivial element of $\Out_{\calF}(P)$ is the class of the conjugation automorphism with any element of $N_D(P)$ which is not in $P$. We may choose this element to be an element of order $2^m$. Note that this element centralizes the cyclic subgroup of order $2^{m-1}$ of $P$. Therefore, under any isomorphism $P\cong L$ the corresponding element of $L$ will again centralize the cyclic subgroup of $L$ of order $2^{m-1}$. It follows that the class of the conjugation with this element is equal to $\epsilon$ as described in the result. 
Finally, by choosing $P$ to be fully $\calF$-centralized, we can see that $C_D(P)=Z(D)$ is a defect group of the group-block pair $(C_G(P),e_P)$. Since $Z(D)=Z(P)$ is central in $C_G(P)$, we obtain $l\big(kC_G(P)e_P\big)=1$. The result now follows from Lemma~\ref{lem 2}, and the fact that $V$ is one-dimensional, since $\Out(L)$ is abelian.

\smallskip
(c) The proof of this part is similar to the proof of Proposition~\ref{prop dihedral}{\rm (c)}. One shows that there is a unique subgroup $P$ of $D$ isomorphic to $L$, that $kC_G(P)e_P$ is nilpotent and hence $l\big(kC_G(P)e_P\big)=1$, and that $\Out_\calF(P)\cong N_D(P)/PC_D(P)\cong C_2$ with the non-trivial element given by the inversion automorphism of $P$. The result follows again from Lemma~\ref{lem 2}. 

\smallskip
(d) We will use Lemma~\ref{lem 4}. By Alperin's fusion theorem, the two conjugacy classes of subgroups of $D$ isomorphic to $Q_8$ are also the $\calF$-isomorphism classes. We may choose $Q_8\cong P\le D$ to be fully $\calF$-centralized. Then $C_D(P)=P$ is a defect group of $(C_G(P), e_P)$ and it is central in $C_G(P)$. Thus, $l\big(kC_G(P)e_P\big)=1$.  The result now follows from  Theorem~\ref{thm mult formula} and Lemma~\ref{lem 4}.

\smallskip
(e) Let $P$ be a subgroup of $D$ isomorphic to $C_4$. We may choose $P$ to be fully $\calF$-centralized. Then $(C_G(P), e_P)$ has a cyclic defect group $C_D(P)$. Since $p=2$, it follows that $kC_G(P)e_P$ is nilpotent and we obtain $l\big(kC_G(P)e_P\big)=1$. Moreover, $\Out_\calF(P)=\Out(P)$. Lemma~\ref{lem 2} implies that $m\big(S_{C_4,1,\FF_-},\FF T^\Delta_{(G,b)}\big)$ is equal to zero and that $m\big(S_{C_4,1,\FF},\FF T^\Delta_{(G,b)}\big)$ is equal to the number of $\calF$-isomorphism classes of subgroups of $D$ isomorphic to $C_4$. The result follows from Remark~\ref{noth quaternion}.

\smallskip
(f) The center $Z=Z(D)$ of $D$ is the unique subgroup of $D$ isomorphic to $C_2$.  Therefore $Z$ is fully $\calF$-centralized and $D$ is a defect group of $(C_G(Z), e_Z)$. Moreover, the fusion system of $(C_G(Z),e_Z)$ is isomorphic to $\calF$. It follows from \cite[Section~3]{Olsson1975} that $l\big(kC_G(Z)e_Z\big)=l(kGb)$.  Theorem~\ref{thm mult formula} and \cite[Corollary~8.23]{BoucYilmaz2022} imply that $m\big(S_{C_2,1,\FF},\FF T^\Delta_{(G,b)}\big)=l\big(kC_G(Z)e_Z\big)=l(kGb)=m\big(S_{\{1\},1,\FF},\FF T^\Delta_{(G,b)}\big)$.  The result follows again from \cite[Section~3]{Olsson1975}.  
\end{proof}

The same reasoning as at the end of Section~4 shows now that Proposition~\ref{prop quaternion}, together with the result in \cite{Yilmaz2024} for $Q_8$, implies the part of Theorem~\ref{thm main} concerning generalized quaternion defect groups.


\section{Blocks with semidihedral defect groups}

Throughout this section we assume that $(G,b)$ is a group-block pair with defect group $D$ isomorphic to $SD_{2^n}=\langle s,t\mid s^{2^{n-1}}=t^2=1,t^{-1}st=s^{2^{n-2}-1}\rangle$, the semidihedral group of order $2^n$, with $n\ge 4$.

\begin{remark}\label{rem Aut semidihedral}
In the group $SD_{2^n}$ as above, for any integer $a$, the order of $ts^a$ is equal to $2$ if $a$ is even and to $4$ if $a$ is odd. Thus any automorphism of $SD_{2^n}$ must map $t$ to an element of the form $ts^a$ with even $a$. It is probably well-known and easy to verify that every automorphism of $SD_{2^n}$ is of the form $f_{a,b}$ with $f_{a,b}(s)=s^b$ and $f_{a,b}(t)=ts^a$ with integers $a,b$ such that $a$ is even and $b$ is odd. One obtains an isomorphism
\begin{equation*}
   2\ZZ/2^{n-1}\ZZ\rtimes(\ZZ/2^{n-1}\ZZ)^\times \to \Aut(SD_{2^n})\,,\quad (\abar,\bbar)\mapsto f_{a,b}\,,
\end{equation*}
where in the semidirect product, the multiplicative group $(\ZZ/2^{n-1}\ZZ)^\times$ acts by multiplication on the additive group $2\ZZ/2^{n-1}\ZZ$. It is straightforward to verify that conjugation with the element $t^ls^k$, for integers $k$ and $l$, is equal to the automorphism $f_{(2^{n-2}-2)k, (2^{n-2}-1)^l}$. 
Since the element $(2^{n-2}-2)+2^{n-1}\ZZ$ generates the additive group $2\ZZ/2^{n-1}\ZZ$, the group $\Inn(SD_{2^n})$ corresponds under the above isomorphism to the subgroup $2\ZZ/2^{n-1}\ZZ\rtimes \langle 2^{n-2}-1+2^{n-1}\ZZ\rangle$. Thus, the above isomorphism induces an isomorphism
\begin{equation*}
   \Out(SD_{2^n}) \cong (\ZZ/2^{n-1}\ZZ)^\times/\langle 2^{n-2}-1+2^{n-1}\ZZ\rangle\,.
\end{equation*}
In particular, $\Aut(SD_{2^n})$ is a $2$-group and $\Out(SD_{2^n})$ is an abelian $2$-group. Note that the element $2^{n-2}-1+2^{n-1}\ZZ$ has order $2$ in the multiplicative group $(\ZZ/2^{n-1}\ZZ)^\times$. It is well-known that the latter group is the direct product of the subgroups generated by the classes of the elements $-1$ and $5$. 
Moreover, the class of $5$ generates a subgroup of index $2$ in $(\ZZ/2^{n-1}\ZZ)^\times$. Thus, $(\ZZ/2^{n-1}\ZZ)^\times$ has precisely 3 elements of order $2$, namely the classes of $-1$, and  $2^{n-2}\pm1$. One can show that in the direct product decomposition
$(\ZZ/2^{n-1}\ZZ)^\times = \langle-1+2^{n-1}\ZZ\rangle \times \langle 5+2^{n-2}\ZZ\rangle$
the element $2^{n-2}-1$ has non-trivial components in both factors. In fact, it is clearly not contained in the first factor and every element of the second factor is congruent to $1$ modulo $4$. This implies that $\Out(SD_{2^n})$ is a cyclic group of order $2^{n-3}$, generated by the class of the automorphism $f_{0,5}$.
\end{remark}

\begin{nothing}\label{noth semidihedral}
(a) Similar to \ref{noth dihedral}(a) and \ref{noth quaternion}(a), we define the following subgroups of $SD_{2^n}$:
\begin{equation*}
   S_m:=\langle s_m\rangle \quad\text{with $s_m:=s^{2^{n-m-1}}$,}
\end{equation*}
a cyclic subgroup of order $2^m$ for $0\le m\le n-1$, and
\begin{equation*}
   T^1_m:=\langle s_{m-1},t\rangle\quad \text{for $1\le m\le n-1$} \quad\text{and}\quad T^2_m:=\langle s_{m-1}, st\rangle\quad \text{for $2\le m\le n-1$.}
\end{equation*}
Then the subgroups $SD_{2^n}$, $S_m$ ($0\le m\le n-1$),  $T^1_m$ ($1\le m\le n-1$) and $T^2_m$ ($2\le m\le n-1$), form a complete set of representatives of the conjugacy classes of subgroups of $Q_{2^n}$. Moreover, $T^1_m$, for $3\le m\le n-1$ (resp. $T^2_m$, for $4\le m\le n-1$),  is isomorphic to a dihedral group (resp. generalized quaternion group) of order $2^{m}$, $T^2_3$ is isomorphic to a quaternion group of order $8$, $T^1_2$ is isomorphic to a Klein-four group and $T^2_2$ is isomorphic to a cyclic group of order $4$. For $i=1,2$ and $2\le m\le n-2$, one has $N_{SD_{2^n}}(T_m^i)=T_{m+1}^i$.

\smallskip
(b) Up to isomorphism, there exist four saturated fusion systems on $D$ (see \cite[Example~3.8]{AKO2011} and \cite[Theorem~5.3]{CravenGlesser2012}). Following again \cite{AKO2011} we denote these fusion systems by $\calF_{ij}$, for $i,j\in\{0,1\}$, where $i=0$ (resp. ~$i=1$) means the fusion system with two (resp. ~one) isomorphism classes of involutions and $j=0$ (resp. ~$j=1$) means the fusion systems with two (resp. ~one) isomorphism classes of cyclic groups of order $4$.  The fusion system $\calF_{00}$, $\calF_{01}$, $\calF_{10}$ and $\calF_{11}$ can be realized, respectively, as the fusion system of the group $D$, of the principal blocks of $\GL(2,q)$ for any odd prime power $q$ such that $(q+1)_2=2^{n-2}$, of the principal blocks of $\PSL(2,q^2)\rtimes C_2$, the non-split extension by the field automorphism, for any odd prime power $q$ such that $(q+1)_2=2^{n-2}$, and of the principal blocks of $\PSL(3,q)$ for any odd prime power $q$ such that $(q+1)_2=2^{n-2}$. One has
\begin{equation*}
   \Aut_{\calF_{0j}}(T_2^1)\cong C_2\,, \text{ and } \ \Aut_{\calF_{1j}}(T_2^1)\cong S_3\, \, \text{for } j=1,2;
\end{equation*}
and
\begin{equation*}
   \Aut_{\calF_{i0}}(T_3^2)\cong C_2\,, \text{ and } \ \Aut_{\calF_{i1}}(T_3^2)\cong S_3\, \, \text{for } i=1,2\,.
\end{equation*}
For each of the four fusion systems, the essential subgroups of $D$ are precisely the subgroups isomorphic to Klein-four groups and to quaternion groups of order $8$.
\end{nothing}

\begin{proposition}\label{prop semidihedral}
Let $(G,b)$ be a group-block pair, let $(D,e_D)$ a maximal $(G,b)$-Brauer pair and let $\calF$ denote the fusion system of $(G,b)$ with respect to $(D,e_D)$.
Suppose that the defect group $D$ of $(G,b)$ is isomorphic to $SD_{2^n}$, the semidihedral group of order $2^n$.

\smallskip
{\rm (a)} Let $L=SD_{2^n}$. Then, for any simple $\FF\Out(L)$-module $V$, one has $m\big(S_{L,1,V},\FF T^\Delta_{(G,b)}\big)=1$. 

\smallskip
{\rm (b)} Let $L=D_{2^m}$ for $3\le m\le n-1$ and let $\epsilon\in \Out(L)$ be the class of the automorphism of $L$ that is the identity on the cyclic subgroup of order $2^{m-1}$ and switches the two conjugacy classes of non-central involutions of $L$. Then for any simple $\FF\Out(L)$-module $V$, one has
\begin{equation*}
   m\big(S_{L,1,V},\FF T^\Delta_{(G,b)}\big) = 
   \begin{cases}
      1\,, & \text{if $\epsilon\in\ker(V)$;}\\
      0\,, & \text{if $\epsilon\notin\ker(V)$.}
   \end{cases}
\end{equation*}

\smallskip
{\rm (c)} Let $L=C_{2^m}$ be a cyclic group of order $2^m$ with $3\le m\le n-1$, and let $V$ be a simple $\FF\Out(L)$-module. Moreover, let $\epsilon\in\Out(L)$ be the class of the inversion automorphism $x\mapsto x^{-1}$ of $L$. Then
\begin{equation*}
   m\big(S_{L,1,V},\FF T^\Delta_{(G,b)}\big) =
   \begin{cases}
      1\,, & \text{if $\epsilon\in\ker(V)$;}\\
      0\,, & \text{if $\epsilon\notin\ker(V)$.}
   \end{cases}
\end{equation*}

\smallskip
{\rm (d)}  Suppose that $n>4$.  Let $L=Q_{2^m}$ with $4\le m\le n-1$ and let $\epsilon\in\Out(L)$ be the class of the automorphism of $L$ that is the identity on the cyclic subgroup of order $2^{m-1}$ and switches the two conjugacy classes of subgroups of order $4$  of $L$ not contained in the cyclic subgroup of order $2^{m-1}$. Then, for any simple $\FF\Out(L)$-module $V$, one has
\begin{equation*}
   m\big(S_{L,1,V},\FF T^\Delta_{(G,b)}\big) = 
   \begin{cases}
      1\,, & \text{if $\epsilon\in\ker(V)$;}\\
      0\,, & \text{if $\epsilon\notin\ker(V)$.}
   \end{cases}
\end{equation*}

\smallskip
{\rm (e)} Let $L=Q_8$ and let $a_1$, $a_2$, $a_3$, $a_4$ denote the multiplicities of $S_{Q_8, 1,\FF}$, $S_{Q_8,1,\FF_-}$, $S_{Q_8,1,V_2}$, $S_{Q_8,u_0,\FF}$ in $\FF T^\Delta_{(G,b)}$, respectively (see Remark~\ref{rem simple functors}(d)). Then
\begin{equation*}
   (a_1,a_2,a_3,a_4)=
   \begin{cases}
      (1,0,1,0)\,, & \text{if $\calF\cong \calF_{00}$;}\\
      (1,0,1,0)\,, & \text{if $\calF\cong \calF_{10}$;}\\
      (1,0,0,1)\,, & \text{if $\calF\cong \calF_{01}$;}\\
      (1,0,0,1)\,, & \text{if $\calF\cong \calF_{11}$.}\\
   \end{cases}
\end{equation*}

\smallskip
{\rm (f)} Let $L=V_4$ and let $a_1$, $a_2$, $a_3$, $a_4$ denote the multiplicities of $S_{V_4, 1,\FF}$, $S_{V_4,1,\FF_-}$, $S_{V_4,1,V_2}$, $S_{V_4,u_0,\FF}$ in $\FF T^\Delta_{(G,b)}$, respectively (see Remark~\ref{rem simple functors}(c)). Then
\begin{equation*}
   (a_1,a_2,a_3,a_4)=
   \begin{cases}
      (1,0,1,0)\,, & \text{if $\calF\cong \calF_{00}$;}\\
      (1,0,0,1)\,, & \text{if $\calF\cong \calF_{10}$;}\\
      (1,0,1,0)\,, & \text{if $\calF\cong \calF_{01}$;}\\
      (1,0,0,1)\,, & \text{if $\calF\cong \calF_{11}$.}\\
   \end{cases}
\end{equation*}

\smallskip
{\rm (g)} Let $L=C_4$. Then
\begin{equation*}
m\big(S_{C_4,1,\FF_-}, \FF T^\Delta_{(G,b)}\big)=0\quad \text{and}\quad  m\big(S_{C_4,1,\FF}, \FF T^\Delta_{(G,b)}\big)=
   \begin{cases}
      2\,, & \text{if $\calF\cong \calF_{00}$;}\\
      2\,, & \text{if $\calF\cong \calF_{10}$;}\\
      1\,, & \text{if $\calF\cong \calF_{01}$;}\\
      1\,, & \text{if $\calF\cong \calF_{11}$.}\\
   \end{cases}
\end{equation*}

\smallskip
{\rm (h)} Let $L=C_2$. Then
\begin{equation*}
m\big(S_{C_2,1,\FF}, \FF T^\Delta_{(G,b)}\big)=
   \begin{cases}
      2\,, & \text{if $\calF\cong \calF_{00}$;}\\
      3\,, & \text{if $\calF\cong \calF_{01}$;}\\
      1\,, & \text{if $\calF\cong \calF_{10}$;}\\
      2\,, & \text{if $\calF\cong \calF_{11}$.}\\
   \end{cases}
\end{equation*}

\smallskip
{\rm (i)} Let $L=\{1\}$. Then
\begin{equation*}
m\big(S_{1,1,\FF}, \FF T^\Delta_{(G,b)}\big)=
   \begin{cases}
      1\,, & \text{if $\calF\cong \calF_{00}$;}\\
      2\,, & \text{if $\calF\cong \calF_{01}$;}\\
      2\,, & \text{if $\calF\cong \calF_{10}$;}\\
      3\,, & \text{if $\calF\cong \calF_{11}$.}\\
   \end{cases}
\end{equation*}
\end{proposition}
\begin{proof}
The proof of this proposition closely follows the arguments in Propositions~\ref{prop dihedral} and~\ref{prop quaternion}: Part~(a) is similar to \ref{prop dihedral}(a) and \ref{prop quaternion}(a); Part~(b) to \ref{prop dihedral}(b); Part~(c) to \ref{prop dihedral}(c) and \ref{prop quaternion}(c); Part~(d) to \ref{prop quaternion}(b); Part~(e) to \ref{prop quaternion}(d); Part~(f) to \ref{prop dihedral}(d); and Part~(g) to \ref{prop quaternion}(e). Since $m\big(S_{1,1,\FF},\FF T^\Delta_{(G,b)}\big)=l(kGb)$, see \cite[Corollary~8.23]{BoucYilmaz2022}, Part~(i) follows from  \cite[Section~3]{Olsson1975}. Finally to prove Part~(h), let $L=C_2$. If $\calF\cong \calF_{0i}$, for $i=1,2$, then there are two $\calF$-isomorphism classes of subgroups of $D$ isomorphic to $C_2$. Let $C_2\cong P\le D$ be a non-central subgroup and $Z=Z(D)$ be the center of $D$. Then $P$ and $Z$ are the representatives of these classes. By Theorem~\ref{thm mult formula}, one has $m\big(S_{C_2,1,\FF},\FF T^\Delta_{(G,b)}\big)=l\big(kC_G(P)e_P\big)+l\big(kC_G(Z)e_Z\big)$. We may choose $P$ to be fully $\calF$-centralized. Then the group-block pair $(C_G(P),e_P)$ has a defect group $C_D(P)\cong V_4$ and the trivial fusion system. It follows that $l\big(kC_G(P)e_P\big)=1$. The group-block pair $(C_G(Z), e_Z)$ has a defect group $D$ and the fusion system $\calF\cong \calF_{0i}$. It follows from \cite[Section~3]{Olsson1975} that $l\big(kC_G(Z)e_Z\big)=l(kGb)$ and the result follows in this case. If $\calF\cong \calF_{1i}$, for $i=1,2$, then there is only one $\calF$-isomorphism class of subgroups of $D$ isomorphic to $C_2$. The center $Z=Z(D)$ is a fully $\calF$-centralized representative of the $\calF$-isomorphism class. It follows that $(C_G(Z), e_Z)$ has a defect group $D$ and the fusion system $\calF_{0i}$. By Theorem~\ref{thm mult formula}, one has $m\big(S_{C_2,1,\FF},\FF T^\Delta_{(G,b)}\big)=l\big(kC_G(Z)e_Z\big)$ and the result follows again from \cite[Section~3]{Olsson1975}. 
\end{proof}

The same reasoning as at the end of Section~4 shows now that Proposition~\ref{prop semidihedral} implies the part of Theorem~\ref{thm main} concerning semidihedral defect groups. The proof of Theorem~\ref{thm main} is now complete.

\medskip
\noindent\textbf{Acknowledgment} \ The first and second authors gratefully acknowledge the hospitality of the Mathematics Department at Bilkent University during visits in 2023, 2024 and 2025. The third author is supported by the Scientific and Technological Research Council of T{\"u}rkiye (T{\"U}B{\.I}TAK) under the 3501 Career Development Program with Project No. 123F456.



\centerline{\rule{5ex}{.1ex}}
\begin{flushleft}
Robert Boltje, Department of Mathematics, University of California, Santa Cruz, 95064, California, USA.\\
{\tt boltje@ucsc.edu} \vspace{1ex}\\
Serge Bouc, CNRS-LAMFA, Universit\'e de Picardie, 33 rue St Leu, 80039, Amiens, France.\\
{\tt serge.bouc@u-picardie.fr}\vspace{1ex}\\
Deniz Y\i lmaz, Department of Mathematics, Bilkent University, 06800 Ankara, Turkey.\\
{\tt d.yilmaz@bilkent.edu.tr}
\end{flushleft}
\end{document}